\documentclass[preprint,12pt]{elsarticle}

\usepackage{amssymb}
\usepackage{amsthm}

\newcommand{\sysn}{\left\{\begin{array}{rcl}}
\newcommand{\sysk}{\end{array}\right.}

\newtheorem{theorem}{Theorem}[section]

\theoremstyle{example}

\theoremstyle{definition}

\journal{...}

\begin{document}

\begin{frontmatter}

\title{On the different kinds of
separability of the space of Borel functions}


\author{Alexander V. Osipov}

\ead{OAB@list.ru}


\address{Krasovskii Institute of Mathematics and Mechanics, Ural Federal
 University, Ural State University of Economics, 620219, Yekaterinburg, Russia}

\begin{abstract} In paper we proved that:

$\bullet$ a space of Borel functions $B(X)$ on a set of reals $X$,
with pointwise topology, to be countably selective sequentially
separable if and only if $X$ has the property $S_{1}(B_{\Gamma},
B_{\Gamma})$;

$\bullet$ there exists a consistent example of sequentially
separable selectively separable space which is not selective
sequentially separable. This is an answer to the question of A.
Bella, M. Bonanzinga and M. Matveev;

$\bullet$ there is a consistent example of a compact $T_2$
sequentially separable space which is not selective sequentially
separable. This is an answer to the question of  A. Bella and C.
Costantini;

$\bullet$  $\min\{\mathfrak{b},\mathfrak{q}\}=\{\kappa:
2^{\kappa}$ is not selective sequentially separable$\}$. This is a
partial answer to the question of A. Bella, M. Bonanzinga and M.
Matveev.


\end{abstract}


\end{frontmatter}



\section{Introduction}

 In \cite{ospy}, Osipov and Pytkeev
gave necessary and sufficient conditions for the space $B_1(X)$ of
first Baire class functions on a Tychonoff space $X$, with
pointwise topology, to be (strongly) sequentially
 separable. In this paper we consider some properties of a space $B(X)$ of
Borel functions  on a set of reals $X$, with pointwise topology,
that are stronger than (sequentially) separability.

\section{Main definitions and notation}

Many topological properties are defined or characterized in terms
 of the following classical selection principles.
 Let $\mathcal{A}$ and $\mathcal{B}$ be sets consisting of
families of subsets of an infinite set $X$. Then:

$S_{1}(\mathcal{A},\mathcal{B})$ is the selection hypothesis: for
each sequence $(A_{n}: n\in \mathbb{N})$ of elements of
$\mathcal{A}$ there is a sequence $(b_{n}: n\in\mathbb{N})$ such
that for each $n$, $b_{n}\in A_{n}$, and $\{b_{n}: n\in\mathbb{N}
\}$ is an element of $\mathcal{B}$.

$S_{fin}(\mathcal{A},\mathcal{B})$ is the selection hypothesis:
for each sequence $(A_{n}: n\in \mathbb{N})$ of elements of
$\mathcal{A}$ there is a sequence $(B_{n}: n\in\mathbb{N})$ of
finite sets such that for each $n$, $B_{n}\subseteq A_{n}$, and
$\bigcup_{n\in\mathbb{N}}B_{n}\in\mathcal{B}$.

$U_{fin}(\mathcal{A},\mathcal{B})$ is the selection hypothesis:
whenever $\mathcal{U}_1$, $\mathcal{U}_2, ... \in \mathcal{A}$ and
none contains a finite subcover, there are finite sets
$\mathcal{F}_n\subseteq \mathcal{U}_n$, $n\in \mathbb{N}$, such
that $\{\bigcup \mathcal{F}_n : n\in \mathbb{N}\}\in \mathcal{B}$.

\medskip

 An open cover $\mathcal{U}$ of a space $X$ is:

 $\bullet$ an {\it $\omega$-cover} if $X$ does not belong to
 $\mathcal{U}$ and every finite subset of $X$ is contained in a
 member of $\mathcal{U}$.

$\bullet$ a {\it $\gamma$-cover} if it is infinite and each $x\in
X$ belongs to all but finitely many elements of $\mathcal{U}$.

For a topological space $X$ we denote:

$\bullet$ $\Omega$ --- the family of all countable open
$\omega$-covers of $X$;

$\bullet$ $\Gamma$ --- the family of all countable open
$\gamma$-covers of $X$;

$\bullet$ $B_{\Omega}$ --- the family of all countable Borel
$\omega$-covers of $X$;

$\bullet$ $B_\Gamma$ --- the family of all countable Borel
$\gamma$-covers of $X$;

$\bullet$ $F_\Gamma$ --- the family of all countable closed
$\gamma$-covers of $X$;

$\bullet$ $\mathcal{D}$ --- the family of all countable dense
subsets of $X$;

$\bullet$ $\mathcal{S}$ --- the family of all countable
sequentially dense subsets of $X$.

\medskip

A $\gamma$-cover $\mathcal{U}$ of co-zero sets of $X$  is {\it
$\gamma_F$-shrinkable} if there exists a $\gamma$-cover $\{F(U) :
U\in \mathcal{U}\}$ of zero-sets of $X$ with $F(U)\subset U$ for
every $U\in \mathcal{U}$.

For a topological space $X$ we denote:

$\bullet$ $\Gamma_F$ --- the family of all countable
$\gamma_F$-shrinkable $\gamma$-covers of $X$.

\medskip

 We will use the following notations.

$\bullet$  $C_{p}(X)$ is the set of all real-valued continuous
functions $C(X)$ defined  on a space $X$, with pointwise topology.

$\bullet$ $B_1(X)$ is the set of all first Baire class
 functions $B_1(X)$ i.e., pointwise limits of continuous functions, defined
 on a space $X$,  with pointwise topology.

$\bullet$  $B(X)$ is the  set of all Borel functions, defined on a
space $X$,  with pointwise topology.

\medskip

  If $X$ is a space and $A\subseteq X$, then the sequential closure of $A$,
 denoted by $[A]_{seq}$, is the set of all limits of sequences
 from $A$. A set $D\subseteq X$ is said to be sequentially dense
 if $X=[D]_{seq}$. If $D$ is a countable sequentially dense subset
 of $X$ then $X$ call {\it sequentially separable} space.

 Call $X$ {\it strongly sequentially separable} if $X$ is separable and
 every countable dense subset of $X$ is sequentially dense.

\medskip
A space $X$ is {\it $($countably$)$ selectively separable} (or
$M$-separable, \cite{bbm}) if for every sequence $(D_n: n\in
\mathbb{N})$ of (countable) dense subsets of $X$ one can pick
finite $F_n\subset D_n$, $n\in \mathbb{N}$, so that $\bigcup
\{F_n: n\in \mathbb{N}\}$ is dense in $X$.

In \cite{bbm} the authors started to investigate a selective
version of sequential separability.

A space $X$ is {\it  $($countably$)$ selectively sequentially
separable} (or $M$-sequentially separable, \cite{bbm}) if for
every sequence $(D_n: n\in \mathbb{N})$ of (countable)
sequentially dense subsets of $X$, one can pick finite $F_n\subset
D_n$, $n\in \mathbb{N}$, so that $\bigcup \{F_n: n\in
\mathbb{N}\}$ is sequentially dense in $X$.

In Scheeper's terminology \cite{sch} countably selectively
separability equivalently to the selection principle
$S_{fin}(\mathcal{D},\mathcal{D})$, and countably selective
sequentially separability equivalently to the
$S_{fin}(\mathcal{S},\mathcal{S})$.

\medskip

Recall that the cardinal $\mathfrak{p}$ is the smallest cardinal
so that there is a collection of $\mathfrak{p}$ many subsets of
the natural numbers with the strong finite intersection property
but no infinite pseudo-intersection. Note that $\omega_1 \leq
\mathfrak{p} \leq \mathfrak{c}$.

For $f,g\in \mathbb{N}^{\mathbb{N}}$, let $f\leq^{*} g$ if
$f(n)\leq g(n)$ for all but finitely many $n$. $\mathfrak{b}$ is
the minimal cardinality of a $\leq^{*}$-unbounded subset of
$\mathbb{N}^{\mathbb{N}}$. A set $B\subset [\mathbb{N}]^{\infty}$
is unbounded if the set of all increasing enumerations of elements
of $B$ is unbounded in $\mathbb{N}^{\mathbb{N}}$, with respect to
$\leq^{*}$. It follows that $|B|\geq \mathfrak{b}$. A subset $S$
of the real line is called a $Q$-set if each one of its subsets is
a $G_{\delta}$. The cardinal $\mathfrak{q}$ is the smallest
cardinal so that for any $\kappa< \mathfrak{q}$ there is a $Q$-set
of size $\kappa$. (See \cite{do} for more on small cardinals
including $\mathfrak{p}$).

\section{Properties of a space of Borel functions}

\begin{theorem}\label{th1} For a set of reals $X$, the following statements are
equivalent:

\begin{enumerate}

\item  $B(X)$ satisfies $S_{1}(\mathcal{S},\mathcal{S})$ and
$B(X)$ is sequentially separable;

\item $X$ satisfies $S_{1}(B_{\Gamma}, B_{\Gamma})$;

\item $B(X)\in S_{fin}(\mathcal{S},\mathcal{S})$ and $B(X)$ is
sequentially separable;

\item $X$ satisfies $S_{fin}(B_{\Gamma}, B_{\Gamma})$;

\item $B_1(X)$ satisfies $S_{1}(\mathcal{S},\mathcal{S})$;

\item $X$ satisfies $S_{1}(F_{\Gamma}, F_{\Gamma})$;

\item $B_1(X)$ satisfies $S_{fin}(\mathcal{S},\mathcal{S})$.

\end{enumerate}

\end{theorem}

\begin{proof}
It is obviously that $(1)\Rightarrow(3)$.

 $(2)\Leftrightarrow(4)$. By Theorem 1 in \cite{scts},  $U_{fin}(B_{\Gamma},
B_{\Gamma})=S_1(B_{\Gamma}, B_{\Gamma})=S_{fin}(B_{\Gamma},
B_{\Gamma})$.

$(3)\Rightarrow(2)$. Let $\{\mathcal{F}_i\}\subset B_\Gamma$ and
$\mathcal{S}=\{h_m\}_{m\in \mathbb{N}}$ be a countable
sequentially dense subset of $B(X)$. For each $i\in \mathbb{N}$ we
consider a countable sequentially dense subset $\mathcal{S}_i$ of
$B(X)$ and $\mathcal{F}_i=\{ F^{m}_i\}_{m\in \mathbb{N}}$ where

$\mathcal{S}_i =\{f^m_i\}:=\{ f^m_i\in B(X) : f^m_i\upharpoonright
F^{m}_i=h_m$ and $f^m_i\upharpoonright (X\setminus F^{m}_i)=1$ for
$m \in \mathbb{N} \}$.

Since $\mathcal{F}_i=\{F^{m}_i\}_{m\in \mathbb{N}}$ is a Borel
$\gamma$-cover of $X$ and $\mathcal{S}$ is a countable
sequentially dense subset of $B(X)$, we have that $\mathcal{S}_i$
is a  countable sequentially dense subset of $B(X)$ for each $i\in
\mathbb{N}$. Indeed, let $h\in B(X)$, there is a sequence
$\{h_s\}_{s\in \mathbb{N}}\subset \mathcal{S}$ such that
$\{h_s\}_{s\in \mathbb{N}}$ converge to $h$. We claim that
$\{f^s_i\}_{s\in \mathbb{N}}$ converge to $h$.  Let
$K=\{x_1,...,x_k\}$ be a finite subset of $X$, $\epsilon>0$ and
let $W=\langle h, K,\epsilon \rangle:=\{g\in B(X):
|g(x_j)-h(x_j)|<\epsilon$ for $j=1,...,k\}$ be a base neighborhood
of $h$, then there is $m_0\in \mathbb{N}$ such that $K\subset
F^{m}_i $ for each $m>m_0$ and $h_s\in W$ for each $s>m_0$. Since
$f^s_i\upharpoonright K= h_s\upharpoonright K$ for every $s>m_0$,
$f^s_i\in W$ for every $s>m_0$. It follows that $\{f^s_i\}_{s\in
\mathbb{N}}$ converge to $h$.

 Since $B(X)$ satisfies $S_{fin}(\mathcal{S},\mathcal{S})$,  there is a sequence $(F_i=\{f^{m_1}_{i},..., f^{m_{s(i)}}_{i}\} : i\in\mathbb{N})$ such
that for each $i$, $F_i \subset \mathcal{S}_i$, and
$\bigcup_{i\in\mathbb{N}}F_{i}$ is a countable sequentially dense
subset of $B(X)$.

For ${\bf 0}\in B(X)$ there is a sequence
$\{f^{m_{s(i_j)}}_{i_j}\}_{j\in \mathbb{N}}\subset
\bigcup_{i\in\mathbb{N}}F_{i}$ such that
$\{f^{m_{s(i_j)}}_{i_j}\}_{j\in \mathbb{N}}$ converge to ${\bf
0}$. Consider a sequence $(F^{m_{s(i_j)}}_{i_j}: j\in
\mathbb{N})$.

(1) $F^{m_{s(i_j)}}_{i_j}\in \mathcal{F}_{i_j}$.

(2) $\{F^{m_{s(i_j)}}_{i_j}: j\in \mathbb{N}\}$ is a
$\gamma$-cover of $X$.

Indeed, let $K$ be a finite subset of $X$ and $U=\langle {\bf 0},
K, \frac{1}{2} \rangle$ be a base neighborhood of ${\bf 0}$, then
there is $j_0\in \mathbb{N}$ such that $f^{m_{s(i_j)}}_{i_j}\in U$
for every $j>j_0$. It follows that $K\subset F^{m_{s(i_j)}}_{i_j}$
for every $j>j_0$. We thus get $X$ satisfies $U_{fin}(B_{\Gamma},
B_{\Gamma})$, and, hence, by Theorem 1 in \cite{scts},  $X$
satisfies $S_{1}(B_{\Gamma}, B_{\Gamma})$.

 $(2)\Rightarrow(1)$.  Let
$\{S_i\}\subset \mathcal{S}$ and $S=\{d_n: n\in \mathbb{N}\} \in
\mathcal{S}$. Consider the topology $\tau$ generated by the family
 $\mathcal{P}=\{f^{-1}(G): G$ is an open set of $\mathbb{R}$ and
 $f\in S\cup \bigcup\limits_{i\in \mathbb{N}} S_i \}$.
 Since $P=S\cup \bigcup\limits_{i\in \mathbb{N}} S_i$ is a countable dense subset
 of $B(X)$ and $X$ is Tychonoff, we have that the space $Y=(X,\tau)$ is a separable metrizable
 space. Note that a function $f\in P$, considered as mapping from
 $Y$ to $\mathbb{R}$, is a continuous function i.e. $f\in C(Y)$ for each $f\in P$. Note also that an identity
 map $\varphi$ from $X$ on $Y$, is a Borel bijection. By Corollary
 12 in \cite{busu}, $Y$ is a $QN$-space and, hence, by Corollary
 20 in \cite{tszd}, $Y$ has the property $S_{1}(B_{\Gamma},
 B_{\Gamma})$. By Corollary 21 in \cite{tszd}, $B(Y)$ is an $\alpha_2$ space.


Let $q:\mathbb{N} \mapsto \mathbb{N}\times \mathbb{N}$ be a
bijection. Then we enumerate $\{S_i\}_{i\in \mathbb{N}}$ as
$\{S_{q(i)}\}_{q(i)\in \mathbb{N}\times \mathbb{N}}$. For each
$d_n\in S$ there are sequences $s_{n,m}\subset S_{n,m}$ such that
$s_{n,m}$ converges to $d_n$ for each $m\in \mathbb{N}$. Since
$B(Y)$ is an $\alpha_2$ space, there is $\{b_{n,m}:
m\in\mathbb{N}\}$ such that for each $m$, $b_{n,m}\in s_{n,m}$,
and, $b_{n,m}\rightarrow d_n$ $(m\rightarrow \infty)$. Let
$B=\{b_{n,m} : n,m\in \mathbb{N} \}$. Note that $S\subset
[B]_{seq}$.

Since $X$ is a $\sigma$-set (that is, each Borel subset of $X$ is
$F_{\sigma}$)(see \cite{tszd}),  $B_1(X)=B(X)$ and
$\varphi(B(Y))=\varphi(B_1(Y))\subseteq B(X)$ where
$\varphi(B(Y)):=\{p\circ \varphi : p\in B(Y)\}$ and
$\varphi(B_1(Y)):=\{p\circ \varphi : p\in B_1(Y)\}$.

Since $S$ is a countable sequentially dense subset of $B(X)$, for
any $g\in B(X)$ there is a sequence $\{g_n\}_{n\in
\mathbb{N}}\subset S$ such that $\{g_n\}_{n\in \mathbb{N}}$
converge to $g$. But $g$ we can consider as a mapping from $Y$
into $\mathbb{R}$ and a set $\{g_n : n\in \mathbb{N}\}$ as subset
of $C(Y)$. It follows that $g\in B_1(Y)$. We get that
$\varphi(B(Y))=B(X)$.

We claim that $B\in \mathcal{S}$, i.e. that $[B]_{seq}=B(X)$. Let
$f\in B(Y)$ and $\{f_k: k\in \mathbb{N}\}\subset S$ such that $f_k
\rightarrow f$ ($k\rightarrow \infty$). For each $k\in \mathbb{N}$
there is $\{f^{n}_k : n\in \mathbb{N} \}\subset B$ such that
$f^{n}_k \rightarrow f_k$ ($n\rightarrow \infty$). Since $Y$ is a
$QN$-space (Theorem 16 in \cite{busu}), there exists an unbounded
$\beta\in \mathbb{N}^{\mathbb{N}}$ such that
$\{f^{\beta(k)}_{k}\}$ converge to $f$ on $Y$. It follows that
$\{f^{\beta(k)}_{k} : k\in \mathbb{N}\}$ converge to $f$ on $X$
and $[B]_{seq}=B(X)$.

 $(5)\Rightarrow(6)$. By Velichko's Theorem (\cite{vel}), a space $B_1(X)$ is sequentially
separable for any separable metric space $X$.

 Let $\{\mathcal{F}_i\}\subset F_\Gamma$ and
$\mathcal{S}=\{h_m\}_{m\in \mathbb{N}}$ be a countable
sequentially dense subset of $B_1(X)$.

Similarly implication $(3)\Rightarrow(2)$ we get $X$ satisfies
$U_{fin}(F_{\Gamma}, F_{\Gamma})$, and, hence, by Lemma 13 in
\cite{tszd}, $X$ satisfies $S_{1}(F_{\Gamma}, F_{\Gamma})$.

$(6)\Rightarrow(5)$. By Corollary 20 in \cite{tszd}, $X$ satisfies
$S_{1}(B_{\Gamma}, B_{\Gamma})$. Since $X$ is a $\sigma$-set (see
\cite{tszd}),  $B_1(X)=B(X)$ and, by implication
$(2)\Rightarrow(1)$, we get  $B_1(X)$ satisfies
$S_{1}(\mathcal{S},\mathcal{S})$.

\medskip

\end{proof}

 In \cite{sch} (Theorem 13) M. Scheepers proved the following result

\begin{theorem}\label{th21} (Scheepers) For $X$ a separable metric space, the
following are equivalent:

\begin{enumerate}

\item $C_p(X)$ satisfies $S_{1}(\mathcal{D},\mathcal{D})$;

\item $X$ satisfies $S_{1}(\Omega, \Omega)$.

\end{enumerate}

\end{theorem}

We claim the theorem for a space $B(X)$ of Borel functions.

\begin{theorem}\label{th22} For a set of reals $X$, the
following are equivalent:

\begin{enumerate}

\item $B(X)$ satisfies $S_{1}(\mathcal{D},\mathcal{D})$;

\item $X$ satisfies $S_{1}(B_\Omega, B_\Omega)$.

\end{enumerate}

\end{theorem}

\begin{proof} $(1)\Rightarrow(2)$. Let $X$ be a set of reals
satisfying the hypotheses and $\beta$ be a countable base of $X$.
Consider a sequence $\{\mathcal{B}_i\}_{i\in \mathbb{N}}$ of
countable Borel $\omega$-covers of $X$ where
$\mathcal{B}_i=\{W^{j}_{i}\}_{j\in \mathbb{N}}$ for each $i\in
\mathbb{N}$.

Consider a topology $\tau$ generated by the family
$\mathcal{P}=\{W^{j}_{i}\cap A: i,j\in \mathbb{N}$ and $A\in
\beta\}\bigcup \{(X\setminus W^{j}_{i})\cap A: i,j\in \mathbb{N}$
and $A\in \beta\} $.

Note that if $\chi_{P}$ is  a characteristic function of $P$ for
each $P\in \mathcal{P}$, then a diagonal mapping
$\varphi=\Delta_{P\in \mathcal{P}} \chi_{P} : X\mapsto 2^{\omega}$
is a Borel bijection. Let $Z=\varphi(X)$.

Note that
 $\{\mathcal{B}_i\}$ is countable open $\omega$-cover of $Z$
 for each $i\in \mathbb{N}$. Since $B(Z)$ is a dense subset of $B(X)$, then
  $B(Z)$ also has the property $S_{1}(\mathcal{D},\mathcal{D})$. Since
  $C_p(Z)$ is a dense subset of $B(Z)$, $C_p(Z)$ has the property $S_{1}(\mathcal{D},\mathcal{D})$, too.

  By Theorem \ref{th21}, the space $Z$ has the property $S_{1}(\Omega,
  \Omega)$. It follows that there is a sequence $\{W^{j(i)}_{i}\}_{i\in
  \mathbb{N}}$ such that $W^{j(i)}_{i}\in \mathcal{B}_i$ and $\{W^{j(i)}_{i} : i\in
  \mathbb{N}\}$ is an open $\omega$-cover of $Z$. It follows that $\{W^{j(i)}_{i} : i\in
  \mathbb{N}\}$ is Borel $\omega$-cover of $X$.

$(2)\Rightarrow(1)$. Assume that $X$ has the property
$S_{1}(B_{\Omega}, B_{\Omega})$. Let $\{D_k\}_{k\in \mathbb{N}}$
be a sequence countable dense subsets of $B(X)$ and $D_k=\{f^k_i :
i\in \mathbb{N}\}$ for each $k\in \mathbb{N}$. We claim that for
any $f\in B(X)$ there is a sequence $\{f_k\}\subset B(X)$ such
that $f_k\in D_k$ for each $k\in \mathbb{N}$ and $f\in
\overline{\{f_k : k\in \mathbb{N}\}}$. Without loss
 of generality we can assume $f=\bf{0}$. For each $f^k_i\in D_k$
 let $W^{k}_i=\{ x\in X
 : -\frac{1}{k}<f^k_i(x)<\frac{1}{k} \}$.

If for each $j\in \mathbb{N}$ there is $k(j)$ such that
$W^{k(j)}_{i(j)}=X$, then a sequence $f_{k(j)}=f^{k(j)}_{i(j)}$
uniformly converge to $f$ and, hence, $f\in \overline{\{f_{k(j)} :
j\in \mathbb{N}\}}$.

We can assume that $W^{k}_i\neq X$ for any $k,i\in \mathbb{N}$.

(a). $\{W^{k}_i\}_{i\in
 \mathbb{N}}$ a sequence of Borel sets of $X$.

(b). For each $k\in \mathbb{N}$, $\{W^{k}_i : i\in
 \mathbb{N}\}$ is a $\omega$-cover of $X$.

By (2),  $X$ has the property $S_{1}(B_{\Omega}, B_{\Omega})$,
hence, there is a sequence $\{W^{k}_{i(k)}\}_{k\in \mathbb{N}}$
such that $W^{k}_{i(k)}\in \{W^{k}_i\}_{i\in
 \mathbb{N}}$  for each $k\in \mathbb{N}$ and $\{W^{k}_{i(k)}\}_{k\in \mathbb{N}}
 $ is a $\omega$-cover of $X$.

 Consider $\{f^{k}_{i(k)}\}$. We claim that $f\in \overline{\{f^{k}_{i(k)} : k\in
 \mathbb{N}\}}$.
Let $K$ be a finite subset of $X$, $\epsilon>0$ and $U=\langle f,
K, \epsilon \rangle$ be a base neighborhood of $f$, then there is
$k_0\in \mathbb{N}$ such that $\frac{1}{k_0}<\epsilon$ and
$K\subset W^{k_0}_{i(k_0)}$. It follows that $f^{k_0}_{i(k_0)}\in
U$.

Let $D=\{d_n: n\in \mathbb{N} \}$ be a dense subspace of $B(X)$.
Given a sequence $\{D_i\}_{i\in \mathbb{N}}$ of dense subspace of
$B(X)$, enumerate it as $\{D_{n,m}: n,m \in \mathbb{N} \}$. For
each $n\in \mathbb{N}$, pick $d_{n,m}\in D_{n,m}$ so that $d_n\in
\overline{\{d_{n,m}: m\in \mathbb{N}\}}$. Then $\{d_{n,m}: m,n\in
\mathbb{N}\}$ is dense in $B(X)$.

\end{proof}

In \cite{sch} (Theorem 35) and \cite{bbm2} (Corollary 2.10)
proved the following result

\begin{theorem}\label{th28} (Scheepers) For $X$ a separable metric space, the
following are equivalent:

\begin{enumerate}

\item $C_p(X)$ satisfies $S_{fin}(\mathcal{D},\mathcal{D})$;

\item $X$ satisfies $S_{fin}(\Omega, \Omega)$.

\end{enumerate}

\end{theorem}

Then for the space $B(X)$ we have analogous result.

\begin{theorem}\label{th24} For a set of reals $X$, the
following are equivalent:

\begin{enumerate}

\item $B(X)$ satisfies $S_{fin}(\mathcal{D},\mathcal{D})$;

\item $X$ satisfies $S_{fin}(B_\Omega, B_\Omega)$.

\end{enumerate}

\end{theorem}

\begin{proof} It is proved similarly to the proof of
Theorem \ref{th22}.
\end{proof}

\section{Question of A. Bella, M. Bonanzinga and M. Matveev}

In \cite{bbm}, Question 4.3, it is asked to find a sequentially
separable selectively separable space which is not selective
sequentially separable.

\medskip

The following theorem answers this question.

\begin{theorem} $(CH)$ There is a consistent example of a space $Z$, such
that  $Z$ is sequentially separable selectively separable, not
selective sequentially separable.
\end{theorem}

\begin{proof} By Theorem 40 and  Corollary 41  in \cite{scts}, there is a $\mathfrak{c}$-Lusin set $X$ which has the property
$S_1(B_{\Omega},B_{\Omega})$, but $X$ does not have the property
$U_{fin}(\Gamma,\Gamma)$.

Consider a space $Z=C_p(X)$.

 By Velichko's Theorem (\cite{vel}), a space $C_p(X)$
is sequentially separable for any separable metric space $X$.

(a). $Z$ is sequentially separable.

Since $X$ is Lindel$\ddot{o}$f and $X$ satisfies
$S_1(B_{\Omega},B_{\Omega})$, $X$ has the property
$S_1(\Omega,\Omega)$.

By Theorem \ref{th21}, $C_p(X)$ satisfies
$S_{1}(\mathcal{D},\mathcal{D})$, and, hence, $C_p(X)$ satisfies
$S_{fin}(\mathcal{D},\mathcal{D})$.

(b). $Z$ is selectively separable.

By Theorem 4.1 in \cite{os11},
$U_{fin}(\Gamma,\Gamma)=U_{fin}(\Gamma_F,\Gamma)$ for
Lindel$\ddot{o}$f spaces.

Since $X$ does not have the property $U_{fin}(\Gamma,\Gamma)$, $X$
does not have the property $S_{fin}(\Gamma_F,\Gamma)$. By Theorem
8.11 in \cite{os2}, $C_p(X)$ does not have the property
$S_{fin}(\mathcal{S},\mathcal{S})$.

(c). $Z$ is not selective sequentially separable.

\end{proof}

\begin{theorem} $(CH)$ There is a consistent example of a space $Z$, such
that  $Z$ is countable sequentially separable selectively
separable, not countable selective sequentially separable.
\end{theorem}

\begin{proof} Consider the $\mathfrak{c}$-Lusin set $X$ (see Theorem 40 and  Corollary 41  in \cite{scts}), then $X$ has the property
$S_1(B_{\Omega},B_{\Omega})$, but $X$ does not have the property
$U_{fin}(\Gamma,\Gamma)$ and, hence, $X$ does not have the
property $S_{fin}(B_{\Gamma},B_{\Gamma})$.

Consider a space $Z=B_1(X)$.

 By Velichko's Theorem in \cite{vel}, a space $B_1(X)$
is sequentially separable for any separable metric space $X$.

(a). $Z$ is sequentially separable.

By Theorem \ref{th22}, $B(X)$ satisfies
$S_{1}(\mathcal{D},\mathcal{D})$. Since $Z$ is dense subset of
$B(X)$ we have that $Z$ satisfies $S_{1}(\mathcal{D},\mathcal{D})$
and, hence, $Z$ satisfies $S_{fin}(\mathcal{D},\mathcal{D})$.

(b). $Z$ is countable selectively separable.

Since $X$ does not have the property
$S_{fin}(B_{\Gamma},B_{\Gamma})$, by Theorem \ref{th1}, $B_1(X)$
does not have the property $S_{fin}(\mathcal{S},\mathcal{S})$.

(c). $Z$ is not countable selective sequentially separable.

\end{proof}

\section{Question of A. Bella and C.
Costantini.}

In \cite{bc}, Question 2.7, it is asked to find a compact $T_2$
sequentially separable space which is not selective sequentially
separable.

\medskip

The following theorem answers this question.

\begin{theorem}\label{th41} $(\mathfrak{b}<\mathfrak{q})$ There is a consistent example of a compact $T_2$
sequentially separable space which is not selective sequentially
separable.
\end{theorem}

\begin{proof}  Let $D$ be a discrete space of size
$\mathfrak{b}$. Since $\mathfrak{b}<\mathfrak{q}$, a space
$2^{\mathfrak{b}}$ is sequentially separable (see Proposition 3 in
\cite{glm}).

We claim that $2^{\mathfrak{b}}$ is not selective sequentially
separable.

On the contrary, suppose that $2^{\mathfrak{b}}$ is selective
sequentially separable. Since $non(S_{fin}(B_{\Gamma},
B_{\Gamma}))=\mathfrak{b}$ (see Theorem 1 and Theorem 27 in
\cite{scts}), there is a set of reals $X$ such that
$|X|=\mathfrak{b}$ and $X$ does not have the property
$S_{fin}(B_{\Gamma}, B_{\Gamma})$. Hence there exists  sequence
$(A_{n}: n\in \mathbb{N})$ of elements of $B_{\Gamma}$  that for
any sequence $(B_{n}: n\in\mathbb{N})$ of finite sets such that
for each $n$, $B_{n}\subseteq A_{n}$, we have that
$\bigcup_{n\in\mathbb{N}}B_{n}\notin B_{\Gamma}$.

 Consider an identity
mapping $id: D \mapsto X$ from the space $D$ onto the space $X$.
Denote $C^i_{n}=id^{-1}(A^i_n)$ for each $A^i_n\in A_n$ and
$n,i\in \mathbb{N}$. Let $C_n=\{C^i_{n}\}_{i\in \mathbb{N}}$ (i.e.
$C_n=id^{-1}(A_n)$) and let $\mathcal{S}=\{h_i\}_{i\in
\mathbb{N}}$ be a countable sequentially dense subset of
$B(D,\{0,1\})=2^{\mathfrak{b}}$.

For each $n\in \mathbb{N}$ we consider a countable sequentially
dense subset $\mathcal{S}_n$ of $B(D,\{0,1\})$  where

$\mathcal{S}_n =\{f^i_n\}:=\{ f^i_n\in B(D,2) :
f^i_n\upharpoonright C^{i}_n=h_i$ and $f^i_n\upharpoonright
(X\setminus C^{i}_n)=1$ for $i \in \mathbb{N} \}$.

Since $C_n=\{C^{i}_n\}_{i\in \mathbb{N}}$ is a Borel
$\gamma$-cover of $D$ and $\mathcal{S}$ is a countable
sequentially dense subset of $B(D,\{0,1\})$, we have that
$\mathcal{S}_n$ is a  countable sequentially dense subset of
$B(D,\{0,1\})$ for each $n\in \mathbb{N}$.

Indeed, let $h\in B(D,\{0,1\})$, there is a sequence
$\{h_s\}_{s\in \mathbb{N}}\subset \mathcal{S}$ such that
$\{h_s\}_{s\in \mathbb{N}}$ converge to $h$. We claim that
$\{f^s_n\}_{s\in \mathbb{N}}$ converge to $h$.  Let
$K=\{x_1,...,x_k\}$ be a finite subset of $D$,
$\epsilon=\{\epsilon_1,...,\epsilon_k\}$ where $\epsilon_j\in
\{0,1\}$ for $j=1,...,k$, and $W=\langle h, K,\epsilon
\rangle:=\{g\in B(D,\{0,1\}): |g(x_j)-h(x_j)|\in \epsilon_j$ for
$j=1,...,k\}$ be a base neighborhood of $h$, then there is a
number $m_0$ such that $K\subset C^{i}_n $ for $i>m_0$ and $h_s\in
W$ for $s>m_0$. Since $f^s_n\upharpoonright K= h_s\upharpoonright
K$ for each $s>m_0$, $f^s_n\in W$ for each $s>m_0$. It follows
that a sequence $\{f^s_n\}_{s\in \mathbb{N}}$ converge to $h$.

 Since $B(D,\{0,1\})$ is selective
sequentially separable, there is a sequence
$\{F_n=\{f^{i_1}_{n},..., f^{i_{s(n)}}_{n}\} : n\in\mathbb{N}\}$
such that for each $n$, $F_n \subset \mathcal{S}_n$, and
$\bigcup_{n\in\mathbb{N}}F_{n}$ is a countable sequentially dense
subset of $B(D,\{0,1\})$.

For $\bf{0}$ $\in B(D,\{0,1\})$ there is a sequence
$\{f^{i_{j}}_{n_j}\}_{j\in \mathbb{N}}\subset
\bigcup_{n\in\mathbb{N}}F_{n}$ such that $\{f^{i_j}_{n_j}\}_{j\in
\mathbb{N}}$ converge to $\bf{0}$. Consider a sequence
$\{C^{i_{j}}_{n_j}: j\in \mathbb{N}\}$.

(1) $C^{i_{j}}_{n_j}\in C_{n_j}$.

(2) $\{C^{i_{j}}_{n_j}: j\in \mathbb{N}\}$ is a $\gamma$-cover of
$D$.

Indeed, let $K$ be a finite subset of $D$ and $U=\langle {\bf 0} ,
K, \{0\} \rangle$ be a base neighborhood of ${\bf 0}$, then there
is a number $j_0$ such that $f^{i_{j}}_{n_j}\in U$ for every
$j>j_0$. It follows that $K\subset C^{i_{j}}_{n_j}$ for every
$j>j_0$. Hence, $\{A^{i_{j}}_{n_j}=id(C^{i_{j}}_{n_j}): j\in
\mathbb{N}\}\in B_{\Gamma}$ in the space $X$, a contradiction.

\end{proof}

Let $\mu=\min\{\kappa: 2^{\kappa}$ is not selective sequentially
separable$\}$. It is well-known that $\mathfrak{p}\leq\mu\leq
\mathfrak{q}$ (see \cite{bbm}).

\begin{theorem}
$\mu=\min\{\mathfrak{b},\mathfrak{q}\}$.
\end{theorem}

\begin{proof} Let $\kappa<\min\{\mathfrak{b},\mathfrak{q}\}$.
Then, by Proposition 3 in \cite{glm}, $2^{\kappa}$ is a
sequentially separable space.

Let $X$ be a set of reals such that $|X|=\kappa$ and  $X$ be a
$Q$-set.

Analogously proof of implication $(2)\Rightarrow(1)$ in Theorem
\ref{th1}, we can claim that $B(X,\{0,1\})=2^X=2^{\kappa}$ is
selective sequentially separable.

It follows that $\mu\geq \min\{\mathfrak{b},\mathfrak{q}\}$.

Since $\mu\leq \mathfrak{q}$ we suppose that $\mu>\mathfrak{b}$
and $\mathfrak{b}<\mathfrak{q}$. Then, by Theorem \ref{th41},
$2^{\mathfrak{b}}$ is not selective sequentially separable. It
follows that $\mu=\min\{\mathfrak{b},\mathfrak{q}\}$.
\end{proof}

In \cite{bbm}, Question 4.12 : is it the case $\mu\in
\{\mathfrak{p}, \mathfrak{q}\}$ ?

\medskip

A partial positive answer to this question is the existence of the
following models of set theory (Theorem 8 in \cite{bmz}):

1. $\mu=\mathfrak{p}=\mathfrak{b}<\mathfrak{q}$;

2.  $\mathfrak{p}<\mu=\mathfrak{b}=\mathfrak{q}$.

and

3. $\mu=\mathfrak{p}=\mathfrak{q}<\mathfrak{b}$.

\medskip
The author does not know whether, in general, the answer can be
negative. In this regard, the following question is of interest.

\medskip

{\bf Question.} Is there a model of set theory in which
$\mathfrak{p}<\mathfrak{b}<\mathfrak{q}$ ?


\bibliographystyle{model1a-num-names}
\bibliography{<your-bib-database>}







\end{document}